\documentclass[12pt]{amsart}
\usepackage{amsmath,amssymb,amsbsy,amsfonts,latexsym,amsopn,amstext,
                                               amsxtra,euscript,amscd,bm,cite,amsthm}

\usepackage{mathtools}  
\usepackage{xfrac} 

\usepackage{lipsum,eso-pic,xcolor}

\usepackage{comment}
\usepackage{bbm}

\newcommand{\Mod}[1]{\ (\mathrm{mod}\ #1)}

\usepackage{url}
\usepackage[colorlinks,linkcolor=blue,anchorcolor=blue,citecolor=blue,backref=page]{hyperref}
\usepackage{color}
\usepackage{float}

\hypersetup{breaklinks=true}

\usepackage[norefs,nocites]{refcheck}

\begin{document}

\newtheorem{theorem}{Theorem}
\newtheorem{lemma}[theorem]{Lemma}
\newtheorem{claim}[theorem]{Claim}
\newtheorem{cor}[theorem]{Corollary}
\newtheorem{prop}[theorem]{Proposition}
\newtheorem{definition}{Definition}
\newtheorem{question}[theorem]{Question}
\newtheorem{remark}[theorem]{Remark}
\newtheorem*{conjecture}{Conjecture}

\newcommand{\hh}{{{\mathrm h}}}
\newcommand*\diff{\mathop{}\!\mathrm{d}}

\numberwithin{equation}{section}
\numberwithin{theorem}{section}
% \numberwithin{remark}{section}
\numberwithin{table}{section}

\def\sssum{\mathop{\sum\!\sum\!\sum}}
\def\ssum{\mathop{\sum\ldots \sum}}

\def\squareforqed{\hbox{\rlap{$\sqcap$}$\sqcup$}}
\def\qed{\ifmmode\squareforqed\else{\unskip\nobreak\hfil
\penalty50\hskip1em\null\nobreak\hfil\squareforqed
\parfillskip=0pt\finalhyphendemerits=0\endgraf}\fi}%%

%  use the AMS-Euler Fraktur fonts
%%%%%%%%%%%%%%%%%%%%%%%%%%%%%%%%%%
\newfont{\teneufm}{eufm10}
\newfont{\seveneufm}{eufm7}
\newfont{\fiveeufm}{eufm5}
%%%%%%%%%%%%%%%%%%%%%%%%%%%%%%%%%
%
%  allow automatic size selection in math mode
%
%%%%%%%%%%%%%%%%%%%%%%%%%%%%%%%%%
\newfam\eufmfam
     \textfont\eufmfam=\teneufm
\scriptfont\eufmfam=\seveneufm
     \scriptscriptfont\eufmfam=\fiveeufm
%%%%%%%%%%%%%%%%%%%%%%%%%%%%%%%%%
%
%  \frak works on a single symbol at a time...
%
\def\frak#1{{\fam\eufmfam\relax#1}}

\newcommand{\bflambda}{{\boldsymbol{\lambda}}}
\newcommand{\bfmu}{{\boldsymbol{\mu}}}
\newcommand{\bfxi}{{\boldsymbol{\xi}}}
\newcommand{\bfrho}{{\boldsymbol{\rho}}}

\newcommand{\bfalpha}{{\boldsymbol{\alpha}}}
\newcommand{\bfbeta}{{\boldsymbol{\beta}}}
\newcommand{\bfphi}{{\boldsymbol{\varphi}}}
\newcommand{\bfpsi}{{\boldsymbol{\psi}}}
\newcommand{\bftheta}{{\boldsymbol{\vartheta}}}

\def\fK{Frak K}
\def\fT{Frak{T}}

\def\fA{{Frak A}}
\def\fB{{Frak B}}
\def\fC{\mathfrak{C}}

\def \balpha{\bm{\alpha}}
\def \bbeta{\bm{\beta}}
\def \bgamma{\bm{\gamma}}
\def \blambda{\bm{\lambda}}
\def \bchi{\bm{\chi}}
\def \bphi{\bm{\varphi}}
\def \bpsi{\bm{\psi}}

\def\eqref#1{(\ref{#1})}

\def\vec#1{\mathbf{#1}}

%\def\squareforqed{\hbox{\rlap{$\sqcap$}$\sqcup$}}
%\def\qed{\ifmmode\squareforqed\else{\unskip\nobreak\hfil
%\penalty50\hskip1em\null\nobreak\hfil\squareforqed
%\parfillskip=0pt\finalhyphendemerits=0\endgraf}\fi}

%%%%%%%%%%%%%%%%%%%%%%%%%
% Alphabet calligraphie %
%%%%%%%%%%%%%%%%%%%%%%%%%
\def\cA{{\mathcal A}}
\def\cB{{\mathcal B}}
\def\cC{{\mathcal C}}
\def\cD{{\mathcal D}}
\def\cE{{\mathcal E}}
\def\cF{{\mathcal F}}
\def\cG{{\mathcal G}}
\def\cH{{\mathcal H}}
\def\cI{{\mathcal I}}
\def\cJ{{\mathcal J}}
\def\cK{{\mathcal K}}
\def\cL{{\mathcal L}}
\def\cM{{\mathcal M}}
\def\cN{{\mathcal N}}
\def\cO{{\mathcal O}}

\def\cQ{{\mathcal Q}}
\def\cR{{\mathcal R}}
\def\cS{{\mathcal S}}
\def\cT{{\mathcal T}}
\def\cU{{\mathcal U}}
\def\cV{{\mathcal V}}
\def\cW{{\mathcal W}}
\def\cX{{\mathcal X}}
\def\cY{{\mathcal Y}}
\def\cZ{{\mathcal Z}}
\newcommand{\rmod}[1]{\: \text{mod} \: #1}

\def\cg{{\mathcal g}}

\def\vr{\mathbf r}

\def\e{{\mathbf{\,e}}}
\def\ep{{\mathbf{\,e}}_p}
\def\em{{\mathbf{\,e}}_m}

\def\Tr{{\mathrm{Tr}}}
\def\Nm{{\mathrm{Nm}\,}}

 \def\SS{{\mathbf{S}}}

\def\lcm{{\mathrm{lcm}}}
\def\ord{{\mathrm{ord}}}

\def\({\left(}
\def\){\right)}
\def\fl#1{\left\lfloor#1\right\rfloor}
\def\rf#1{\left\lceil#1\right\rceil}

\def\mand{\qquad \text{and} \qquad}

\newcommand{\commM}[1]{\marginpar{%
\begin{color}{red}
\vskip-\baselineskip %raise the marginpar a bit
\raggedright\footnotesize
\itshape\hrule \smallskip M: #1\par\smallskip\hrule\end{color}}}

\newcommand{\commI}[1]{\marginpar{%
\begin{color}{magenta}
\vskip-\baselineskip %raise the marginpar a bit
\raggedright\footnotesize
\itshape\hrule \smallskip I: #1\par\smallskip\hrule\end{color}}}

\newcommand{\commK}[1]{\marginpar{%
\begin{color}{blue}
\vskip-\baselineskip %raise the marginpar a bit
\raggedright\footnotesize
\itshape\hrule \smallskip K: #1\par\smallskip\hrule\end{color}}}

%%%%%%%%%%%%%%%%%%%%%%%%%%%%%%%%%%%%%%%%%%%%%%%%%%%%%%%%
%%%%%%%%%%%%%%%%%%%%%%%%%%%%%%%%%%%%%%%%%%%%%%%%%%%%%%%%
%%%%%%%%%%%%%%%%%%%%%%%%%%%%%%%%%%%%%%%%%%%%%%%%%%%%%%%%
%%%%%%%%%%%%%%%%%%%%%%%%%%%%%%%%%%%%%%%%%%%%%%%%%%%%%%%%

%%%%%%%  END OF STANDARD STUFF %%%%%%%%%

%%%%%%%%%%%%%%%%%%%%%%%%%%%%%%%%%%%%%%%%%%%%%%%%%%%%%%%%
%%%%%%%%%%%%%%%%%%%%%%%%%%%%%%%%%%%%%%%%%%%%%%%%%%%%%%%%
%%%%%%%%%%%%%%%%%%%%%%%%%%%%%%%%%%%%%%%%%%%%%%%%%%%%%%%%
%%%%%%%%%%%%%%%%%%%%%%%%%%%%%%%%%%%%%%%%%%%%%%%%%%%%%%%
%%%%%%%%%%%
%%% Spell

\hyphenation{re-pub-lished}

\mathsurround=1pt

\def\bfdefault{b}
\overfullrule=5pt

\def \F{{\mathbb F}}
\def \K{{\mathbb K}}
\def \N{{\mathbb N}}
\def \Z{{\mathbb Z}}
\def \Q{{\mathbb Q}}
\def \R{{\mathbb R}}
\def \C{{\mathbb C}}
\def\Fp{\F_p}
\def \fp{\mathfrak p}
\def \fq{\mathfrak q}

\def\ZK{\Z_K}

\def \xbar{\overline x}
\def\e{{\mathbf{\,e}}}
\def\ep{{\mathbf{\,e}}_p}
\def\eq{{\mathbf{\,e}}_q}

%\title[Congruences with exponential polynomials]{Congruences with exponential polynomials}

\title[Product of an integer \& prime in arithmetic progression]{Product of an integer free of small prime factors and prime in arithmetic progression}
\date{\today}

\author[K. H. Yau]{Kam Hung Yau}

\address{Department of Pure Mathematics, University of New South Wales,
Sydney, NSW 2052, Australia}
\email{kamhungyau.math@gmail.com}

\begin{abstract}    
We establish estimates for the number of ways to represent any reduced residue class as a product of a prime and an integer free of small prime factors. Our best results is conditional on the Generalised Riemann hypothesis (GRH). As a corollary, we make progress on a conjecture of Erd\"os, Odlyzko, and S\'ark\"ozy. 
\end{abstract}


\maketitle

\section{Introduction}

We are interested in the following conjecture stated in a paper of Erd\"os, Odlyzko, and S\'ark\"ozy~\cite{EOS}.

\begin{conjecture}[EOSC]
For all sufficiently large $k$ and $a$ with $(a,k)=1$, we have
\begin{equation} \label{eq:eosc}
p_1 p_2 \equiv a \Mod{k} 
\end{equation}
for some primes $p_1, p_2 \le  k$.
\end{conjecture}

Although the EOSC is unreachable with current methods in view of the parity problem, we note that various relaxation towards the EOSC had been made. Specifically, Shparlinski~\cite{S} showed for any integers $a$ and $k \ge 1$ with $(a,k)=1$,  there exists several families of small integers $j, \ell \ge 1$ and real positive $\alpha, \beta \le 1$, such that the following holds
$$
p_1 \ldots p_j s \equiv a \Mod{k},
$$
where $p_1, \ldots ,p_j \le k^{\alpha}$ are primes and $s \le k^{\beta}$ is a product of at most $\ell$ primes. Shparlinski~\cite{S1} also showed that there exist a solution to the congruence 
$$
pr \equiv a \Mod{k},
$$
where $p$ is prime, $r$ is a product of at most $17$ prime factors and $p,r \le k$. The techniques in~\cite{S1,S} involves a sieve method by Greaves~\cite{Gr} applied with bounds of exponential sum over reciprocal of primes.

For products of large primes, Ramar\'e and Walker~\cite{RW} showed that every reduced residue class modulo $k$ can be represented by a product of three primes $p_1,  p_2, p_3 \le k^4$ as $k \rightarrow \infty$. More recently, Klurman, Mangerel, and Ter\"av\"ainen~\cite{KMT} have stated that under the Generalised Riemann Hypothesis (GRH) we get
$$
\max_{a \in \mathbb{Z}_k^{\times}} \min \{ n \in \mathbb{N} : n \equiv a \Mod{k}, n \in E_3  \} \ll k^{2+o(1)},
$$
where $E_3$ is the set of numbers that are product of exactly three primes.

In another direction Friedlander, Kurlberg, Shparlinski~\cite{FKS} obtained an upper bound on the number of solutions to~\eqref{eq:eosc} on average over $a$ and $k$. This suggests we should expect the following conjecture 
$$
\# \{ (p_1, p_2): \mbox{primes } p_i \le x , p_1 p_2 \equiv a \Mod{k} \} = \frac{\pi (x)^2}{\varphi(k)} + o \left (\frac{\pi(x)^2}{\varphi(k)} \right ),
$$
where $\pi(x)$ denote the number of primes up to $x$.

Finally, we remark that one can find results of Garaev~\cite{G} which improve results of~\cite{FKS} concerning the related congruences 
$$
p_1(p_2 + p_3) \equiv a \Mod{k}    
$$
and 
$$
p_1 p_2 (p_3 + h) \equiv a \Mod{k},
$$
where $p_1,  p_2, p_3 \le x$ are primes and $h$ is a fixed integer.

 Let $\cU \subseteq [2,x] \cap \mathbb{N}$, $\cV \subseteq [2,y] \cap \mathbb{N}$, $k>2$, and $a$ with $(a,k)=1$. We denote by $\cN_{k}(a; \cU, \cV,z)$ the number of solutions to the congruence
$$
u v \equiv a \Mod{k}, \quad u \in \cU, v \in \cV, P^-(u) \ge z.
$$
Here $P^-(m)$ is the smallest prime factor of $m$ for $m \ge 2$ with the usual convention $P^-(1) = \infty$. 

In the special case when $\cU  = [2,x] \cap \mathbb{N}$ and $\cV = \{p \le y : p \mbox{ prime}\}$, we set
$$
\cN_{k}(a; x, y,z) = \cN_{k}(a; \cU, \cV,z).
$$

Observe that showing
$$
\cN_{k}(a; k, k,k^{\frac{1}{2} })  >0
$$
for all sufficiently large $k$ would immediately imply the EOSC.

In this paper we establish various bounds for $\cN_{k}(a; x, y,z)$ where our best results are conditional on the GRH. Our strategy is to apply the Harman sieve coupled with  Type I and II estimates obtained from bounds for multiplicative character sums.

\section{Notation}
We recall the notation $A=O(B)$ and $A \ll B$ are all equivalent to the assertion that the inequality $|A| \le cB$ holds for some positive absolute  constant $c$. Consequently, we write $A \asymp B$ to mean both $A \ll B$ and $B \ll A$. We denote
$m \sim M$ to mean integers satisfying $M\le m < 2M$.

 We write $\chi_0$ to be the principal character modulo $k$ and the set of all $\varphi(k)$ multiplicative character modulo $k$ is denoted by  $\mathcal{X}_k$, where $\varphi$ is the Euler totient function. Moreover, we denote $\mathcal{X}_k^*= \mathcal{X}_k \backslash \{\chi_0 \}$.

For relatively prime integers $k$ and $n$, we denote by $\bar{n}_k$ the multiplicative inverse of $n$ modulo $k$, the unique integer $u$ defined by the conditions
$un\equiv 1 \Mod{k}$ with $0 \le u < k$. We always denote $p,q$ and their subscripts to be prime.

\section{Main results}

For any $\beta \ge 0$, we denote
$$
\Phi_k(x, x^{\beta}) = \# \{  n \le x : P^-(n) \ge x^{\beta}, (n,k)=1 \}
$$
as the number of $x^{\beta}$-rough numbers in the interval $[1,x]$ coprime to $k$, and
$$
\mathcal{P}_k(y) = \{ p \le y : (p,k)=1 \}
$$
as the set of all primes up to $y$ coprime to $k$. In the special case $k=1$, we write $\Phi(x, x^{\beta}) = \Phi_1(x, x^{\beta})$ and $\mathcal{P}(y) =\mathcal{P}_1(y)$.

 We state our first result for $\cN_{k}(a; x, y , z )$ which is unconditional on the GRH.

\begin{theorem} \label{uncond result}
Let $k, x \ll \log^B y$ for some fixed $B >0$. Then for any $\beta \in (0, \frac{1}{2} ]$ and fixed $C>0$, we have
$$
\cN_{k}(a; x, y ,x^{\beta})= \frac{\# \mathcal{P}_k(y)}{\varphi(k)}   \Phi_k(x, x^{\beta}) + O\bigg (\frac{xy}{\varphi(k) \log^C y}   \bigg ).
$$
\end{theorem}

Assuming the GRH, we obtain an estimate valid for a wider range of parameters.

\begin{theorem} \label{GRH result}
Assume the GRH. Fix real numbers $\vartheta_1, \vartheta_2 >0$ such that 
$$
\vartheta_2 < \min \bigg \{ \frac{1 + \vartheta_1}{2}, \frac{2 + 3\vartheta_1}{5} \bigg \}.
$$  Set $y=x^{\vartheta_1}$, $k \asymp x^{\vartheta_2}$ and fix $\beta \in (0, \frac{1}{2}]$ with $\beta < 1 +2(\vartheta_1 - \vartheta_2 -2\varepsilon)$ for any fixed sufficiently small $\varepsilon >0$. Then we have
$$
\cN_k(a;x,y ,x^{\beta}) = \frac{  \# \mathcal{P}_k(y)}{\varphi(k)} \Phi_k(x, x^{\beta}) +O \bigg ( \frac{x^{1- \varepsilon  + o(1) } y}{k} \bigg )
$$
as $x \rightarrow \infty$.
\end{theorem}

Notice that  Theorem~\ref{uncond result} and~\ref{GRH result} gives the following partial results toward the EOSC, that is, for $(a,k)=1$ with $k$ sufficiently large, we have
$$
\cN_k(a;k, \exp(k^{\frac{1}{B}}) ,k^{ \frac{1}{2} }) >0.
$$
Moreover, assuming the GRH we obtain
$$
\cN_k(a; k, k^{1+\varepsilon}, k^{\frac{1}{2}}) >0.
$$

We see from Theorem~\ref{GRH result} that even on the assumption of the GRH we need one of the two lengths $x$ or $y$ to be greater than the modulus $k$. In view of the EOSC and focusing on the special case $\beta= 1/2$, our next result shows that we can reduce the length of $y$ drastically.

\begin{theorem} \label{GRH lower bound}
Assume the GRH. For any fixed sufficiently small $\varepsilon >0$, set $y=x^{\varepsilon}$ and $k \asymp x^{\delta}$. For all sufficiently large $x$, we have
$$
\cN_k(a;x, y , x^{ \frac{1}{2} })  = \frac{  \# \mathcal{P}_k(y)}{\varphi(k)} \Phi(x, x^{\frac{1}{2} }) +O \bigg ( \frac{x^{1- \varepsilon  + o(1) } y}{k} \bigg )
$$
if $\delta \in [\frac{1}{4}, \frac{1}{3})$, and
\begin{equation} \label{lower bound}
\cN_k(a;x, y , x^{ \frac{1}{2} })  \ge   \frac{( 0.0342   +o(1)) xy}{\varphi(k) (\log x) (\log y)} 
\end{equation}
if $\delta \in [\frac{1}{3}, \frac{2}{5})$. 
\end{theorem}

By~\cite[Theorem 7.11]{MV}, it follows that we have for any fixed $\beta <1$, the asymptotic
$$
\Phi(x, x^{\beta}) \sim  \omega (\beta^{-1}) \frac{x}{ \log x^\beta}
$$
as $x \rightarrow \infty$. Here $\omega$ is the Buchstab function defined by the delay differential equation
\[
\begin{cases}
\omega(u) = 1/u   & \mbox{ for $1 \le u \le 2$,} \\
(u \omega(u))'   = \omega(u-1) &  \mbox{ for $u >2$.} 
\end{cases}
\]

For $k = y^{O(1)}$, the prime number theorem implies
$$
\mathcal{P}_k(y) = \mathcal{P}(y) - \sum_{p|k}1 =\mathcal{P}(y) +O(\log y)    \sim \frac{y}{\log y},
$$
as $y \rightarrow \infty$. It follows the main term dominates the remainder term in Theorems~\ref{uncond result}, \ref{GRH result}, and~\ref{GRH lower bound}. We do not pursue to optimise the constant 0.0342  in~\eqref{lower bound}.

We remark that equation~\eqref{lower bound} of Theorem~\ref{GRH lower bound} implies for $(a,k) = 1$, there exists primes $p_1,p_2$ with $p_1 \le k^{\varepsilon}$, $p_2 \le k^{5/2 + \varepsilon}$ such that
$$
p_1 p_2 \equiv a \ (\mathrm{mod} \ k)
$$
when $k$ is sufficiently large.

\section{Preparations}

\subsection{Bounds for multiplicative character sums}

We recall a classical result independently proved by P\'olya and Vinogradov~\cite[Theorem 12.5]{IK}.

\begin{lemma}[P\'olya-Vinogradov] \label{PV bound}
\overfullrule=0pt For any non-principal character modulo $k$, we have
$$
\max_{M,N} \bigg |\sum_{M < n \le M + N} \chi(n) \bigg | \ll k^{\frac{1}{2} +o(1)}.
$$
\end{lemma}

We also recall a result from~\cite[Corollary 5.29]{IK} which gives a bound for character sums over primes.

\begin{lemma} \label{PNT bound}
For $k >2$ and fixed $A>0$, we have
$$
 \max_{\substack{\chi \in \mathcal{X}_k^*}} \bigg |\sum_{p \le y} \chi(p) \bigg |\ll k^{\frac{1}{2}} y (\log y)^{-A}.
$$
\end{lemma}

We obtain a stronger bound under the GRH. This follows by taking $T=x$ in (13) on page 120 of~\cite{D} and applying the GRH.

\begin{lemma} \label{GRH bound}
Assume the GRH then we have
$$
\max_{\substack{\chi \in \mathcal{X}_k^* }} \bigg |\sum_{p \le y} \chi(p) \bigg | \ll y^{ \frac{1}{2} } \log k y.
$$
\end{lemma}

We also recall the mean value estimate for character sums which follows immediately by orthogonality.

\begin{lemma} \label{mean value bound}
For $N \ge 1$ and any sequence of complex numbers $a_n$, we have
$$
\sum_{\chi \in \mathcal{X}_k} \bigg | \sum_{n \le N} a_n \chi(n) \bigg |^2 \le \varphi(k)(N/k+1) \sum_{n \le N} |a_n|^2.
$$
\end{lemma}

\subsection{Type I and II estimates}

We recall that $k > 2$  is an integer and we define the sequences $\cA = ( c_r)$ by
$$
c_{r} = \sum_{\substack{p \le y, (p,k)=1 \\ r \equiv a \bar{p}_k \Mod{k}  }} 1,
$$
and $\cB = (\mathbbm{1}_{(r,k)=1})$ both supported on the interval $[1,x]$.

\begin{lemma}[Type I estimate] \label{type I}
Suppose we have the bound
$$
\sum_{\substack{p \le y }} \chi(p) \ll \Delta(k,y)
$$
for all $\chi \in \mathcal{X}_k^* $.
For any complex sequence $a_m$ such that $a_m \ll M^{o(1)}$, we have
\begin{equation} \label{eqn:type I}
\sum_{\substack{c_{mn} \in \cA \\ m \le M }} a_m c_{mn} = \frac{ \# \mathcal{P}_k(y)}{\varphi(k)} \sum_{\substack{mn \in \cB \\ m \le M  }} a_m   + O \left (\Delta(k,y) M^{1+o(1)} k^{\frac{1}{2} +o(1)}  \right  ).
\end{equation}
\end{lemma}

\begin{proof}
We recall the orthogonality relation
$$
\frac{1}{\varphi(k)} \sum_{\chi \in \mathcal{X}_k} \bar{\chi}(a)\chi(r) =
\begin{cases}
1 &\mbox{if $r \equiv a \Mod{k}$,} \\
0 &\mbox{otherwise,}
\end{cases}
$$
for $(a,k)=1$. Applying the above identity, we get
\begin{align*}
\sum_{\substack{c_{mn} \in \cA \\  m \le M }} a_m c_{mn} & =  \sum_{\substack{mn \in \cB \\  m \le M }} a_m  \sum_{\substack{ p \le y, (p,k)=1 \\ mn \equiv a \bar{p}_k \Mod{k} }} 1 \\
& = \frac{1}{\varphi(k)}  \sum_{\substack{mn \in \cB \\  m \le M }} a_m \sum_{\substack{ p \le y \\ (p,k)=1 } }  \sum_{\chi \in \mathcal{X}_k}  \chi(mn) \bar{\chi}(a \bar{p}_k).
\end{align*}

Separating the main term corresponding to the principal character $\chi_0$, the above becomes
$$
\frac{\# \mathcal{P}_k(y)}{\varphi(k)}  \sum_{\substack{  mn \in \cB \\  m \le M }} a_m  +  \frac{1}{\varphi(k)}\sum_{\substack{\chi \in \mathcal{X}_k^*  }} \bar{\chi} (a) \sum_{\substack{ mn \le x \\ m \le M}} a_m \chi(mn) \sum_{\substack{ p \le y \\(p,k)=1}} \chi(p).
$$

Denote the second sum on the right by $R$. By the P\'olya-Vinogradov inequality  (Lemma \ref{PV bound}) and our assumption, we obtain
\begin{align*}
R & \ll M^{1+o(1)} \sum_{\substack{ \chi \in \mathcal{X}_k^* }}  \max_{m \le M} \bigg  |\sum_{\substack{  n \le x/m}}  \chi(n) \bigg |  \cdot  \bigg | \sum_{ \substack{ p \le y  }}  \chi(p) \bigg  |   \\
& \ll M^{1+o(1)}  \varphi(k) k^{\frac{1}{2} +o(1)} \Delta(k,y).
\end{align*}
\end{proof}

Using similar argument to Lemma~\ref{type I}, we obtain our Type II estimate.

\begin{lemma}[Type II estimate] \label{type II}
Suppose we have the bound
$$
\sum_{\substack{p \le y }} \chi(p) \ll \Delta(k,y)
$$
for all $\chi \in \mathcal{X}_k^* $.
For any complex sequences $a_m, b_n$ such that $a_m, b_n \ll x^{o(1)}$, we have
\begin{align*} 
\sum_{\substack{c_{mn} \in \cA \\ m \sim M }} a_m b_n c_{mn}  = & \frac{\# \mathcal{P}_k(y)  }{\varphi(k)}  \sum_{\substack{mn \in \cB \\ m \sim M  }}  a_m  b_n   \\
&   +  O \Big ( \Delta(k,y) \Big ( \frac{x}{k} + \frac{M^{\frac{1}{2}} x^{\frac{1}{2}}}{k^{\frac{1}{2}}} + \frac{x}{M^{\frac{1}{2}} k^{\frac{1}{2}}} + x^{\frac{1}{2}} \Big ) (xM)^{o(1)} \Big).
\end{align*}
\end{lemma}

\begin{proof}
We proceed as in the proof of Lemma~\ref{type I} and it is enough to bound
$$
R=\sum_{\substack{\chi \in \mathcal{X}_k^* }}  \sum_{\substack{  mn \le x \\ m \sim M}} a_m b_n   \chi(mn) \sum_{\substack{p  \le y \\ (p,k)=1 }} \bar{\chi}(a\bar{p}_k).
$$
We apply
$$
\sum_{t \le x/m} \frac{1}{\lfloor \frac{x}{M} \rfloor} \sum_{z=1}^{\lfloor \frac{x}{M} \rfloor} e \bigg ( \frac{z(t-n)}{\lfloor x/M \rfloor} \bigg ) =
\begin{cases}
1 &\mbox{if $n \le x/m$,} \\
0 &\mbox{otherwise,}
\end{cases}
$$
to $R$ in order to separate the dependence on $m$ in the summation over $n$. Indeed $R$ becomes \overfullrule=0pt
\begin{align*}
&   \sum_{\substack{\chi \in \mathcal{X}_k^* }}  \sum_{\substack{  n \le x/M \\ m \sim M}} \left (\sum_{t \le x/m} \frac{1}{\lfloor \frac{x}{M} \rfloor} \sum_{z=1}^{\lfloor \frac{x}{M} \rfloor} e \bigg ( \frac{z(t-n)}{\lfloor \frac{x}{M} \rfloor} \bigg ) \right ) a_m b_n   \chi(mn) \sum_{\substack{p  \le y \\ (p,k)=1 }} \bar{\chi}(a\bar{p}_k)\\
& =  \frac{1}{\lfloor \frac{x}{M} \rfloor} \sum_{z=1}^{\lfloor \frac{x}{M} \rfloor}  \sum_{\substack{\chi \in \mathcal{X}_k^* }}  \sum_{\substack{  n \le x/M \\ m \sim M}} \left (\sum_{t \le x/m} e \bigg ( \frac{zt}{\lfloor \frac{x}{M}  \rfloor} \bigg ) \right ) a_m b_n e \bigg ( \frac{-zn}{\lfloor \frac{x}{M}  \rfloor} \bigg )\chi(mn) \sum_{\substack{p  \le y \\ (p,k)=1 }} \bar{\chi}(a\bar{p}_k).  \\
\end{align*}

Recall from~\cite[Bound~(8.6)]{IK}, we have
$$
\sum_{t \le x/m} e \bigg ( \frac{zt}{\lfloor x/M \rfloor} \bigg ) \ll \frac{\lfloor x/M \rfloor}{1 + \min \left \{z, \lfloor x/M \rfloor -z \right \} }
$$
for $z= 1 , \ldots,  \lfloor x/M \rfloor -1$. We apply this to obtain
$$
R \ll \log (x) \sum_{\substack{\chi \in \mathcal{X}_k^* }} \Bigg |   \sum_{\substack{n \le x/M \\ m \sim M}} a_m^* b_n^*   \chi(mn) \Bigg | \cdot \Bigg | \sum_{\substack{ p \le y \\ (p,k)=1 }} \bar{\chi}(a\bar{p}_k) \Bigg |,
$$
where $|a_m^*|=|a_m|, |b_n^*|=|b_n|$. By Cauchy's inequality and Lemma~\ref{mean value bound}, we bound
\begin{align*}
R & \ll (\log x)\max_{ \substack{\chi \in \mathcal{X}_k^* }} \Bigg |\sum_{ \substack{p \le y }} \chi(p) \Bigg | \cdot  \sum_{\substack{\chi \in \mathcal{X}_k}} \Bigg |\sum_{m \sim M} a_m^* \chi(m) \Bigg| \cdot \Bigg|\sum_{n \le x/M} b_n^* \chi(n) \Bigg| \\
& \ll (\log x) \Delta(k,y) \Bigg( \sum_{\substack{\chi \in \mathcal{X}_k}} \Bigg|\sum_{m \sim M} a_m^* \chi(m)\Bigg|^2  \cdot \sum_{\chi \in \mathcal{X}_k} \Bigg |\sum_{n \le x/M} b_n^* \chi(n)  \Bigg |^2 \Bigg)^{ \frac{1}{2} } \\
& \ll  \Delta(k,y) \bigg (\varphi(k) \bigg  ( \frac{M}{k} +1 \bigg )M \varphi(k) \bigg (\frac{x}{Mk} +1 \bigg ) \frac{x}{M} \bigg )^{\frac{1}{2}} (xM)^{o(1)} \\
& \ll  \Delta(k,y) \varphi(k)    \bigg ( \frac{x}{k} + \frac{M^{\frac{1}{2}} x^{\frac{1}{2}}}{k^{\frac{1}{2}}} + \frac{x}{M^{\frac{1}{2}} k^{\frac{1}{2}}} + x^{\frac{1}{2}} \bigg  ) (xM)^{o(1)},
\end{align*}
where we have used the bound $a_m, b_n \ll x^{o(1)}$.
\end{proof}

\subsection{Harman sieve}

In this section, we set $\cA = (\xi_r)$ and $\cB = (\eta_r)$ to be any general sequence of complex numbers supported on $[1 , x]$. For any positive integer $s$, we denote the sequence
$$
\cA_s = (\xi_{rs}).
$$

Moreover, for any positive real number $z$, we define the weighted sifting function to be
$$
\cS(\cA,z) = \sum_{\substack{ r \le x \\ (r,P(z))=1}} \xi_r,
$$
where $P(w) = \prod_{p < w } p$ is the product of all primes less than $w$. We direct the interested reader to Harman's monograph on sieves~\cite{H2} for more information. 

We recall a lemma which is essentially due to Buchstab~\cite[Eq. (13.58)]{IK} but we state it here with weighted sifting function, see also~\cite{H1,H,H2}.

\begin{lemma}[Buchstab identity] \label{buchstab identity}
For any $0 < z_2 \le z_1$, we have
$$
\cS(\cA,z_1) = \cS(\cA, z_2) - \sum_{ z_2 \le p <z_1} \cS(\cA_p,p).
$$
\end{lemma}

It is easy to see that the following variant of Harman sieve follows closely from~\cite[Lemma 2]{H}. 

\begin{prop}[Harman sieve] \label{Harman sieve}
Suppose that for any $a_m, b_n \ll x^{o(1)}$, we have for some $\lambda >0$, $\alpha >0$, $\beta \le \frac{1}{2}$, $M \ge 1$, that
\begin{equation} \label{eq:type I (harman)}
\sum_{\substack{\xi_{mn} \in \cA \\ m \le M}} a_m \xi_{mn} = \lambda \sum_{\substack{\eta_{mn} \in \cB \\ m \le M}} a_m \eta_{mn}  +O(Y)
\end{equation}
and
\begin{equation} \label{eq:type II (harman)}
\sum_{\substack{\xi_{mn} \in \cA \\ x^{\alpha} \le m \le x^{\alpha + \beta}}} a_m b_n \xi_{mn} = \lambda\sum_{\substack{ \eta_{mn}  \in \cB \\ x^{\alpha} \le m \le x^{\alpha + \beta}}} a_m b_n \eta_{mn}   + O(Y).
\end{equation}
Then, if $|c_r| \le 1$, $x^{\alpha} < M$, $R < \min \{x^{1-\alpha}, M \}$ and $M \ge x^{1-\alpha}$ if $R > x^{\alpha + \beta}$, we have
$$
\sum_{r \sim R} c_r \cS(\cA_r, x^{\beta})= \lambda \sum_{r\sim R} c_r \cS(\cB_r, x^{\beta}) +O(Y\log^3(x)).
$$
\end{prop}

Lastly, we recall a result on rough numbers~\cite[Theorem 7.11]{MV}.

\begin{lemma}[Buchstab function] \label{count:buch}
For $ u \ge 1$, we have
$$
\# \{n \le x : P^-(n) \ge x^{1/u} \}     \sim     \frac{x}{\log x} u \omega(u),
$$
where $\omega$ is the Buchstab function defined by the delay differential equation
$$ 
\begin{cases}
\omega(u) = 1/u &\mbox{ for $1 \le u \le 2$,} \\
(u \omega(u))' = \omega(u-1) & \mbox{ for $u >2$.}
\end{cases}
$$
\end{lemma}

\section{Proof of Theorem~\ref{uncond result}}
We recall that $k > 2$ is an integer and we define the sequences $\cA = ( c_r)$ by
$$
c_{r} = \sum_{\substack{p \le y, (p,k)=1 \\ r \equiv a \bar{p}_k \Mod{k}  }} 1,
$$
and $\cB = (\mathbbm{1}_{(r,k)=1})$ both supported on the interval $[1,x]$.

Let $B$ and $C$ be as stated in Theorem~\ref{uncond result} and let $A = 3B +C +3$. By Lemma~\ref{PNT bound}, we assert
$$
\max_{\substack{\chi \in \mathcal{X}_k^*}} \bigg |\sum_{p \le y} \chi(p) \bigg |\ll  \frac{k^{\frac{1}{2}}y}{(\log y)^{A}}.
$$
With
$$
\Delta (k,y) = \frac{k^{\frac{1}{2} }y}{ (\log y)^{A} },
$$
we obtain Type I estimate by applying Lemma~\ref{type I} in order to get
$$
\sum_{\substack{c_{mn} \in \cA \\ m \le M }} a_m c_{mn} = \frac{ \# \mathcal{P}_k(y)}{\varphi(k)} \sum_{\substack{mn \in \cB \\ m \le M }} a_m   + O (R_1 ).
$$
Here
$$
R_1 =  \frac{M^{1+o(1)} k^{1 +o(1)} y}{ (\log y)^{3B+C+3} } \ll \frac{M^{1+o(1)}y}{(\log y)^{B+C+3}} \ll \frac{x y}{ \varphi(k) \log^{C+3} y},
$$
whenever $M \ll x^{1-\varepsilon}$ for any fixed $\varepsilon >0$ and $x$ sufficiently large. We have used the assumption $k \ll \log^B y$.

To obtain our Type II estimate, we apply Lemma~\ref{type II}  at most $\log x$ times to get
\begin{align*} 
\sum_{\substack{c_{mn} \in \cA \\ m \le M }} a_m b_n c_{mn}  = & \frac{\# \mathcal{P}_k(y)}{\varphi(k)}  \sum_{\substack{mn \in \cB \\ m \le M }}  a_m  b_n   +  O  (R_2 ),
\end{align*}
where
\begin{align*}
R_2 & =  \frac{k^{1/2} y}{ (\log y)^{3B+C +3 }} \bigg ( \frac{x}{k} + \frac{M^{\frac{1}{2}} x^{\frac{1}{2}}}{k^{\frac{1}{2}}} + \frac{x}{M^{\frac{1}{2}} k^{\frac{1}{2}}} + x^{\frac{1}{2}} \bigg ) (xM)^{o(1)}  \\
&  \ll  \frac{x y}{ \varphi(k) \log^{C+3} y}
\end{align*}
whenever $1 \ll M \ll x$.

In view of applying Proposition~\ref{Harman sieve}, let us set the following parameters:
\begin{align*}
\lambda &= \# \mathcal{P}_k(y)/\varphi(k), \\
\alpha &=\mbox{ sufficiently small  $\varepsilon >0$} , \\
\beta & \in (0, 1/2 ],  \\
M & = x^{1-\varepsilon},  \\
Y & = \frac{xy}{ \varphi(k) \log^{C+3} y},  \\
 R & =1   \mbox{ and $c_1=1$}. 
\end{align*}

Considering the above, we obtain our Type I and II estimate~\eqref{eq:type I (harman)} and~\eqref{eq:type II (harman)}.
Clearly $x^{\alpha} < M$ as $\varepsilon$ is sufficiently small, therefore by appealing to Harman sieve (Proposition~\ref{Harman sieve}) we get
$$
\cS(\cA,x^{\beta}) = \lambda \cS(\cB, x^{\beta}) + O (Y \log^3 x    )
$$
and the result follows.

\section{Proof of Theorem~\ref{GRH result}}

We recall that $k > 2$ is an integer and we define the sequences $\cA = ( c_r)$ by
$$
c_{r} = \sum_{\substack{p \le y, (p,k)=1 \\ r \equiv a \bar{p}_k \Mod{k}  }} 1,
$$
and $\cB = (\mathbbm{1}_{(r,k)=1})$ both supported on the interval $[1,x]$.

By Lemma~\ref{GRH bound}, we have
$$
\max_{\substack{\chi \in \mathcal{X}_k^*}} \bigg |\sum_{p \le y} \chi(p) \bigg | \ll  y^{ \frac{1}{2} } \log(k y).
$$
Set
$$
\Delta (k,y) = y^{ \frac{1}{2} } \log(k y).
$$
By Lemma~\ref{type I}, we get our Type I estimate
$$
\sum_{\substack{c_{mn} \in \cA \\ m \le M }} a_m c_{mn} = \frac{ \# \mathcal{P}_k(y)}{\varphi(k)} \sum_{\substack{mn \in \cB \\ m \le M }} a_m   + O (R_1 ),
$$
where
$$
R_1 = y^{\frac{1}{2}} \log(ky)  M^{1+o(1)} k^{\frac{1}{2} +o(1)} \ll  \frac{ x^{1 - \varepsilon + o(1)} y}{ \varphi(k) },
$$
whenever $M \ll x^{1+\frac{1}{2}(\vartheta_1 - 3\vartheta_2) -\varepsilon +o(1) }$.

By Lemma~\ref{type II}, we get estimate Type II
$$
\sum_{\substack{c_{mn} \in \cA \\ m \sim M }} a_m b_n c_{mn} = \frac{ \# \mathcal{P}_k(y)}{\varphi(k)} \sum_{\substack{mn \in \cB \\ m \sim M }} a_m b_n  + O (R_2 ).
$$
Here
\begin{align*}
R_2 & = y^{ \frac{1}{2} } (\log k y) \bigg ( \frac{x}{k} + \frac{M^{\frac{1}{2}} x^{\frac{1}{2}}}{k^{\frac{1}{2}}} + \frac{x}{M^{\frac{1}{2}} k^{\frac{1}{2}}} + x^{\frac{1}{2}} \bigg ) (xM)^{o(1)}.
\end{align*}
Since $2\vartheta_2 < 1 + \vartheta_1$, we obtain
$$
R_2 \ll \frac{x^{1 - \varepsilon + o(1)} y}{\varphi(k)},
$$
whenever
$$
 x^{\vartheta_2 - \vartheta_1 + 2\varepsilon +o(1) } \ll M \ll x^{1+\vartheta_1 -\vartheta_2 -2 \varepsilon +o(1)}.
$$

In view of applying Proposition~\ref{Harman sieve}, let us set the following parameters:
\begin{align*}
\lambda & = \# \mathcal{P}_k(y)/\varphi(k), \\
 \alpha & = \vartheta_2 - \vartheta_1 + 2\varepsilon +o(1),\\
  \beta & \in (0, 1/2 ] \mbox{ with $\beta <  1 + 2(\vartheta_1 - \vartheta_2 - 2 \varepsilon$)}, \\
  M &= x^{1+\frac{1}{2}(\vartheta_1 - 3\vartheta_2) -\varepsilon +o(1) }, \\
   Y & = x^{1-\varepsilon +o(1) }y k^{-1}, \\
   R& =1  \mbox{ and $c_1=1$.} 
\end{align*}

\overfullrule=0pt Considering the above, we obtain our Type I and II estimate~\eqref{eq:type I (harman)} and~\eqref{eq:type II (harman)}.
Our assumption $5\vartheta_2 < 2 + 3\vartheta_1$ implies $x^{\alpha} < M$ for $x$ sufficiently large. Therefore by the Harman sieve (Proposition~\ref{Harman sieve}) we assert
$$
\cS(\cA,x^{\beta}) = \lambda \cS(\cB, x^{\beta}) + O  (  Y  ),
$$
where the $\log^3 x$ term is absorbed into $Y$.

\section{Proof of Theorem~\ref{GRH lower bound}}

We recall that $k > 2$ is an integer and we define the sequences $\cA = ( c_r)$ by
$$
c_{r} = \sum_{\substack{p \le y, (p,k)=1 \\ r \equiv a \bar{p}_k \Mod{k}  }} 1,
$$
and $\cB = (\mathbbm{1}_{(r,k)=1})$ both supported on the interval $[1,x]$.

Let $\varepsilon' >0$ with $\varepsilon = 3\varepsilon '$ so that $y = x^{\varepsilon} =x^{3\varepsilon'}$.
By the proof of Theorem~\ref{GRH result} (take $\varepsilon = \varepsilon'$ there), we have satisfactory Type I estimate as long as 
$$
M \ll x^{1 - \frac{3}{2} \delta}.
$$
Moreover, the Type II estimate remains valid when
$$
 x^{\delta} \ll  M \ll x^{1 -\delta}.
$$

Let us set 
$$\lambda = \# \mathcal{P}_k(y)/\varphi(k)$$
and 
$$Y = \frac{ x^{1 - \varepsilon + o(1)} y}{ k}.$$

$\bullet$ Assume $\delta \in \left [ \frac{1}{4} , \frac{1}{3} \right  )$.

Write $X=x^{\frac{1}{2}}$, $z = x^{1- 2\delta}$ and note that $z \le X$. Applying the Buchstab identity (Lemma~\ref{buchstab identity}), we assert
\begin{align*}
\cS( \cA, X ) & =  \cS(\cA, z) - \sum_{ z \le  p < X} \cS(\cA_p,p)   = \Sigma_1 - \Sigma_2, \mbox{ say}.
\end{align*}

By the Harman sieve, we get
$$
\Sigma_1 = \lambda \cS(\cB,z) +O(Y). 
$$

We have $x^{\delta} \le z < X \le x^{1 -\delta}$ since $\delta <1/2$, and $z^3 > x$ since $\delta < 1/3$. Therefore we may write $\Sigma_2$ as a Type II sum and  obtain
\begin{align*}
\Sigma_2 & = \sum_{z \le p < X} \cS(\cA_p,z)  = \lambda \sum_{z \le p < X} \cS(\cB_p, z ) +O(Y).
\end{align*}
Hence in total, we get
$$
\cS( \cA, X ) = \lambda \cS( \cB, X ) + O(Y).
$$
Note that  $\cS(\cB,X) = \Phi_k(x,x^{1/2}) = \Phi(x,x^{1/2})$ since the $k \ll x^{1/3}$.\newline

$\bullet$ Assume $\delta \in [ \frac{1}{3} , \frac{2}{5} )$.
	
	 Write $X=x^{\frac{1}{2}}$, $z = x^{1- 2\delta}$, $T= x^{\delta}$. Then by Buchstab identity
\begin{align*}
\cS( \cA, X ) & =  \cS(\cA, z) - \sum_{ z \le  p < X} \cS(\cA_p,p) \\
& =  \cS(\cA, z) - \sum_{z \le p < T} \cS(\cA_p,p) - \sum_{T \le p < X} \cS(\cA_p,p) \\
& = \Sigma_1 - \Sigma_2 - \Sigma_3 \mbox{, say}.
\end{align*}

The sums $\Sigma_1$ and $ \Sigma_3$ can be estimated as above. For $\Sigma_2$, we apply the Buchstab identity to get
$$
\Sigma_2 = \sum_{z \le p < T} \cS(\cA_p,z) - \sum_{\substack{ z \le q < p < T }} \cS(\cA_{pq},q) = \Sigma_4 - \Sigma_5, \mbox{ say}.
$$

The sum $\Sigma_4$ can be estimated by Harman sieve since $\delta <2/5$. We split $\Sigma_5$ and write
$$
\Sigma_5 = \sum_{\substack{ z \le q < p < T \\ pq \le x^{1-\delta} }} \cS(\cA_{pq},q) + \sum_{\substack{ z \le q < p < T \\ pq > x^{1-\delta} }} \cS(\cA_{pq},q) = \Sigma_6 + \Sigma_7, \mbox{ say}.
$$
Since $T \le z^2 \le pq \le x^{1- \delta }$, the sum $\Sigma_6$ can be estimated as a Type II sum. Indeed let us write
\begin{align*}
\Sigma_6 & = \sum_{\substack{ z \le q < p < T \\ pq \le x^{1-\delta} }}  \sum_{\substack{rpq\le x \\ (r,P(q))=1 }} c_{rpq} \\
& = \sum_{\substack{c_{mn} \in \cA \\ x^{\delta} < m \le x^{1-\delta} }}  \Bigg ( \sum_{\substack{m=pq \\ z \le q < p < T}}  \mathbbm{1}_{P^{-}(n) > q-0.5}   \Bigg ) c_{mn}.
\end{align*}
Recall the truncated Perron formula~\cite[Lemma 2.2]{H2} for
$$
\frac{1}{\pi} \int_{-T}^{T} e^{i \gamma t} \frac{\sin \rho t}{t} \diff t =
\begin{cases}
1 + O (T^{-1} (\rho - |\gamma|)^{-1} ) & \mbox{ if $|\gamma| < \rho$,} \\
O (T^{-1} (\rho - |\gamma|)^{-1} )  & \mbox{ if $|\gamma| > \rho$.} 
\end{cases}
$$
Applying this with $\rho =\log  P^-(n)$, $\gamma = \log (q-0.5)$, $T=x^2 \lambda^{-1}$, we get
$$
\Sigma_6 = M + R,
$$
where
\begin{align*}
M & = \frac{1}{\pi} \int_{-T}^{T}    \sum_{\substack{c_{mn} \in \cA \\ x^{\delta} < m \le x^{1-\delta} }}  \Bigg ( \sum_{\substack{m=pq \\ z \le q < p < T}}   e^{i\gamma t}   \Bigg ) \frac{ \sin(\rho t)}{t} c_{mn} \diff t, \\
R & \ll \frac{x^{o(1)}}{T}\sum_{\substack{c_{mn} \in \cA \\ x^{\delta} < m \le x^{1-\delta} }}  \frac{ c_{mn}}{|\log (P^-(n)) - \log (q-0.5)|}.
\end{align*}

By the mean value theorem
$$
\frac{1}{|\log (P^-(n)) - \log(q-0.5)|}  = \frac{\eta}{| P^-(n) - q - 0.5|},
$$
where $\eta \in [P^-(n) , q-0.5]$ or $\eta \in [q-0.5, P^-(n)]$. In any case
$$
\frac{1}{|\log P^-(n) - \log q|}   \ll \max \{n,q\} \le x.
$$
Recall that
$$
 c_{mn} = \sum_{\substack{p \le y, (p,k)=1 \\ mn \equiv a \bar{p}_k \Mod{k} }} 1 \le y  = x^{\varepsilon},
$$
and therefore $R$ is bounded by
$$
 \ll \frac{1}{T} x^{2+\varepsilon+o(1)} \ll \lambda x^{\varepsilon + o(1)}.
$$

We return our attention to $M$. Note that the integral in the main term between $-1/T$ and $1/T$ can be trivially bounded by 
$$
\ll T^{-1} x^{1+ \varepsilon +o(1)} \ll \lambda.
$$
 Applying our Type II estimate over the region 
$$
\mathcal{R}(T) = (-T,-1/T) \cup (1/T,T),
$$
we obtain
\begin{align*}
M &= \frac{\lambda}{\pi} \int_{\mathcal{R}(T)}   \sum_{\substack{mn \in \cB \\ x^{\delta} < m \le x^{1-\delta} }}  \Bigg ( \sum_{\substack{m=pq \\ z \le q < p < T}}  e^{i\gamma t}   \Bigg ) \frac{ \sin(\rho t)}{t} \diff t + O\left (Y\int_{1/T}^{T} \frac{\diff t}{t} + \lambda  \right ) \\
 & = \frac{\lambda}{\pi} \int_{-T}^T   \sum_{\substack{mn \in \cB \\ x^{\delta} < m \le x^{1-\delta} }}  \Bigg ( \sum_{\substack{m=pq \\ z \le q < p < T}}   e^{i\gamma t}   \Bigg ) \frac{ \sin(\rho t)}{t} \diff t + O(Y ) \\
 & = \lambda \sum_{\substack{ z \le q < p < T \\ pq \le x^{1-\delta} }} \cS(\cB_{pq},q)   + O(Y),
\end{align*}
where the last line follows from the truncated Perron formula. Therefore in total we have
$$
\cS( \cA, X )  = \lambda( \Sigma_1^* - \Sigma_3^* - \Sigma_4^* + \Sigma_6^*) + \Sigma_7 +O(Y),
$$
where $\Sigma_j^*$ is $\Sigma_j$ with $\cA$ replaced by $\cB$. We drop by positivity of $\Sigma_7$ and obtain
\begin{align} \label{S lower bound}
\cS( \cA, X )  & \ge \lambda(\Sigma_1^* - \Sigma_3^* - \Sigma_4^* + \Sigma_6^* + \Sigma_7^* - \Sigma_7^* )  +O(Y) \\
& = \lambda \cS( \cB, X ) - \lambda \Sigma_7^* +O(Y). \nonumber
\end{align}

We note that if $pq^2 >  x$ then $\Sigma_7^*$ is zero, therefore we write
$$
\Sigma_7^* =\sum_{\substack{ z \le q < p < T \\ x^{1-\delta}/q < p \le x/q^2  }} \cS(\cB_{pq},q).
$$
Consider the summand in $\Sigma_7^*$. Choose $u$ so that $q = (x/pq)^{1/u}$ and it follows  by Lemma~\ref{count:buch} 
\begin{align*}
\cS( \cB_{pq}, q ) & \sim \frac{x}{pq \log q} \omega \left (  \frac{\log (x/pq)}{\log q} \right ) \le \frac{x}{pq \log q}.
\end{align*}
since the Buchstab function is bounded by 1.

Hence
\begin{align*}
\Sigma_7^*  
& \le (1+o(1))  \sum_{\substack{ z \le q < p < T \\ x^{1-\delta}/q < p \le x/q^2  }}  \frac{x}{pq \log q} \\
& = (1+o(1))  x \sum_{z \le q < T} \frac{1}{q\log q} \sum_{\max \{q,x^{1-\delta}/q \} < p < \min\{ T, x/q^2\} } \frac{1}{p}  \\
&  \le (1+o(1)) x \sum_{z \le q < T} \frac{ \log \log(x^{2/5}) - \log \log(q)  }{q\log q},
\end{align*}
where the last line follows from Mertens estimate and that $\delta \in [1/3,2/5)$. By partial summation, we find
\begin{align*}
\Sigma_7^*  &  \le (1+o(1)) x \int_{z}^{T} \frac{ \log\log (x^{2/5}) - \log\log (q)  }{q (\log q)^2} \diff q \\
 & = (1+o(1))  \frac{x}{\log x} \int_{1-2\delta}^{\delta} \frac{\log(2/5) - \log(\alpha)}{\alpha^2} \diff \alpha \\
& \le (1+o(1))  \frac{x}{\log x} (\log(32) - 5/2)
\end{align*}
by the substitution $q = x^{\alpha}$ with $(\log x) \diff \alpha  = \diff q/q$, and noting that $\delta \in [1/3,2/5)$.

Since
$
\cS(\cB,X) = (1+o(1))x/ \log(x)
$
for $k \ll x^{2/5}$, and appealing to~\eqref{S lower bound}, we have
\begin{align*}
\cS(\cA,X) & \ge \lambda \frac{x}{\log x} (1 - \log(32) +5/2 +o(1) )  +O(Y) \\
& \ge  \frac{ (0.0342 +o(1)) xy}{ \varphi(k) (\log x )(\log y)}
\end{align*}
as $x \rightarrow \infty$.

\section*{Acknowledgement}

The author thanks I. E. Shparlinski and L. Zhao for helpful comments. This work is supported by an Australian Government Research Training Program (RTP) Scholarship, UNSW PhD Writing Scholarship, and the Lift-off Fellowship of AustMS. The author is also grateful to the referee for their excellent comments which improved the presentation of the article.

\end{document}